\documentclass{amsart}
\usepackage{amsfonts}
\usepackage{color}
\usepackage{graphicx}
\usepackage[bookmarksnumbered,colorlinks, linkcolor=blue, citecolor=red, pagebackref, bookmarks, breaklinks]{hyperref}

\newtheorem{thm}{Theorem}[section]

\newtheorem{lem}[thm]{Lemma}
\newtheorem{prop}[thm]{Proposition}
\theoremstyle{definition}

\theoremstyle{remark}

\newtheorem{rem}[thm]{Remark}

\numberwithin{equation}{section}
\newcommand{\R}{\mathbb R}
\newcommand{\N}{\mathbb N}
\newcommand{\x}{\tilde{x}}
\newcommand{\ut}{\widetilde{u}}
\newcommand{\tpsi}{\widetilde{\psi}}

\newcommand{\gfls}{(-\Delta_g)^s}
 
\newcommand{\ve}{\varepsilon}
\newcommand{\lam}{\lambda}

\title[Maximum principles, Liouville theorem and symmetry results]{Maximum principles, Liouville theorem and symmetry results for the fractional $g-$Laplacian}

\author{Sandra Molina, Ariel Salort and Hern\'an Vivas}

\address[SM]{Centro Marplatense de Investigaciones matem\'aticas, CIC-UNMdP, Mar del Plata, Argentina}

\email{smolina@mdp.edu.ar}

\address[AS]{Instituto de C\'alculo, UBA-CONICET, Buenos Aires, Argentina \\
Departamento de Matem\'atica, FCEN - Universidad de Buenos Aires\\
Ciudad Universitaria, Pabell\'on I, C1428EGA, Av. Cantilo s/n\\
Buenos Aires, Argentina}

\email{asalort@dm.uba.ar}

\address[HV]{Instituto de C\'alculo, UBA-CONICET, Buenos Aires, Argentina \\
Centro Marplatense de Investigaciones matem\'aticas, CIC-UNMdP, Mar del Plata, Argentina}

\email{havivas@mdp.edu.ar}

\begin{document}



\begin{abstract}
We study different maximum principles for non-local non-linear operators with non-standard growth that arise naturally in the context of fractional Orlicz-Sobolev spaces and whose most notable representative is the fractional $g-$Laplacian:
\[
\gfls u(x):=\textrm{p.v.}\int_{\R^n}g\left(\frac{u(x)-u(y)}{|x-y|^s}\right)\frac{dy}{|x-y|^{n+s}},
\]
being $g$ the derivative of a Young function. 

We further derive qualitative properties of solutions such as a Liouville type theorem and symmetry results and present several possible extensions and some interesting open questions. These are the first results of this type proved in this setting. 
\end{abstract}

\maketitle
\tableofcontents
 
\section{Introduction and main results}

The aim of this manuscript is to study qualitative properties of the so-called fractional $g-$Laplacian; for $s\in(0,1)$ the fractional $g-$Laplacian is defined by
\[
\gfls u(x):=\textrm{p.v.}\int_{\R^n}g\left(\frac{u(x)-u(y)}{|x-y|^s}\right)\frac{dy}{|x-y|^{n+s}}
\]
where p.v. stands for the principal value and $g=G'$ is the derivative of a Young function $G$ (for this and other definitions see Section \ref{sec.prel}).  

This operator was introduced in \cite{FBS} and has received an increasing attention in the last years since it allows to model non-local problems obeying a non-power behavior. See for instance \cite{ACPS, ACPS2, BO, BS, bondersalortvivas, DNFBS, SV, S} and the references therein.

We will be particularly interested in different versions of  \emph{maximum principles} for the fractional $g-$Laplacian, from where many  qualitative properties of solutions will be obtained.

The literature on maximum principles and the consequential qualitative properties of solutions (such as symmetry, for instance) is nowadays huge, and different techniques were developed in order to overcome technical difficulties arisen by the particular nature of the operators under study. For instance, the square power case (i.e. $G(t)=t^2$) corresponds with the \emph{fractional Laplacian}, and several  tools such as representation formulas for solutions, the Fourier transform or the Caffarelli-Silvestre  extension method  have shown to be useful, and a series of successful results have been obtained (see \cite{CZ, CHL, MN} and the references therein). However, when $G(t)=t^p$ 
(the well-known  \emph{fractional $p-$Laplacian}) these effective techinques no longer work  due to the nonlinearity of the operator and new techniques and ideas need to be developed. Several results regarding maximum principles and qualitative properties of solutions were proved in \cite{CL, CW, WRBH, WL, ZH,  ZAW}, just to mention some recent works. Furthermore, the method of moving planes introduced in the celebrated paper by Gidas, Ni  and Nirenberg \cite{GNN}, or the sliding method developed by Berestycki and Nirenberg \cite{BN,BN2} provide a flexible alternative to approach symmetry and related issues, and have been adapted to the nonlocal setting in the upper cited papers, among others.

In this manuscript we have as main goal to introduce several formulations of the maximum principle for the fractional $g-$Laplacian, from where we will deduce some interesting qualitative results such as a Liouville type theorem, or symmetry of solutions in a ball. 

The novelty of our results is that they can be applied to non-local operators admitting behaviors more general than powers such as $G(t)=t^p\log(1+t)$, $p\geq 2$, or models related to double phase problems where $G(t)=t^p + t^q$, where $p,q\geq 2$. To the best of the authors' knowledge, these are the first results of this kind available in the literature for non-standard growth models.

We further highlight that the possible lack of homogeneity of $G$ will be one of the main obstacles to overcome, and  leads us to develop specific tools for this setting. 


Throughout the paper, solutions of equations involving the fractional $g-$Laplacian will be assumed to be of class $C^{1,1}_{\text{loc}}(\R^n)\cap L_g$, being $L_g$ the tail space defined as
$$
L_g:=\left\{u\in L^1_{\text{loc}}(\R^n) :\int_{\R^n}g\left(\frac{|u(x)|}{1+|x|^s}\right)\frac{dx}{1+|x|^{n+s}}<\infty\right\}.
$$
That regularity ensures the operator to be well-posed, see Lemma \ref{puntual}.

Our first result is a rather standard maximum principle for the fractional $g-$Laplacian:
\begin{thm}[Maximum principle on domains] \label{teo0}
Let $\Omega$ be a bounded domain in $\R^n$. Assume that  $u\in C^{1,1}_{\text{loc}}(\R^n)\cap L_g$ and satisfies
\begin{align*}
\begin{cases}
\gfls u(x)\geq 0 &\text{ if } x\in\Omega\\
u(x)\geq 0 &\text{ if } x\in\R^n\setminus \Omega.
\end{cases}
\end{align*}
Then $u\geq 0$ in $\Omega$. Moreover, if $u(x)=0$ at some point $x\in\Omega$, then $u\equiv 0$ in $\R^n$.
\end{thm}

In our next result we prove that if  $u$ is a bounded $g-$harmonic function, then it is symmetric about any given hyper-plane in $\R^n$
and hence it must be constant:

\begin{thm}[Liouville] \label{liou}
Let $u\in C^{1,1}_{\text{loc}}(\R^n)\cap L_g$ satisfying
\[
\gfls u =0\quad\text{ in }\R^n.
\]
If $u$ is bounded, then $u$ is constant in $\R^n$. 
\end{thm}

The idea in order to  obtain symmetry of $u$ with respect to a given hyper-plane $H$ is to consider the function $w(x) = u(\tilde x)-u(x)$, where $\tilde x$ denotes for the reflection of $x$ with respect to $H$. If we can prove that $w(x)\leq 0$ in $H$, then interchanging the roles of $x$ and $\tilde x$, we could deduce that $w(x) \equiv 0$ in $H$, and therefore $u(x)$ would be symmetric with respect to the plane $H$. Since the fractional $g-$Laplacian is invariant under rotations and translations this gives that $u$ must be constant.

The aforementioned strategy is reached by means of the following maximum principle for antisymmetric functions on hyperplanes:
\begin{thm}[Maximum principle on hyperplanes] \label{maxpplehyp}
Let $H$ be a hyperplane in $\R^n$, $\Sigma$ the half space at one side of the plane and $\x$ be the reflection of $x$ across $H$. Let $u\in C^{1,1}_{\text{loc}}(\R^n)\cap L_g$ and define
\[
\ut(x):=u(\x),\quad\text{and}\quad w(x):=\ut(x)-u(x).
\]
Assume $w$ is bounded in $\Sigma$. If for any $x\in\Sigma$ such that $\ut(x)>u(x)$ we have
\begin{equation}\label{eq.ineq}
\gfls \ut(x)-\gfls u(x)\leq 0
\end{equation}
then
\begin{equation}\label{eq1}
w(x)\leq 0\text{ in }\Sigma.
\end{equation} 
\end{thm}

Theorems like \ref{liou} are often generalized to allow some growth at infinity on the function $u$; indeed, the classical Liouville theorem for harmonic functions states that 
\begin{equation}\label{eq.liouh}
\Delta u=0 \quad \textrm{in}\quad\R^n\quad\textrm{and}\quad u=O(|x|^k)\quad\textrm{as}\quad |x|\rightarrow\infty
\end{equation}
imply that $u$ is a polynomial of order at most $k$. Even if the techniques displayed here do not seem to be adaptable to get a result under assumptions similar to \eqref{eq.liouh}, the problem is interesting and worth pointing out.

We will also be interested in studying nonlinear equations of the form
\[
\gfls u=f(u)
\]
under suitable assumptions on the nonlinearity $f$. The classical method in this scenario is the method of moving planes; before stating the results we introduce some notation (which is fairly standard): let $\lambda\in\R$ and
$$
T_\lam:=\{x\in\R^n\colon x_n=\lam \text{ for } \lam \in \R\}
$$
be the hyperplane at height $\lambda$; let
\[
\Sigma_\lam := \{x\in\R^n\colon x_n<\lam\}
\]
be the upper half-space. For each $x\in \Sigma_\lam$ let
$$
x^\lam:=(x_1,x_2,\cdots, 2\lam-x_n)
$$
be its reflection about the plane $T_\lam$. Finally, we will denote 
\[
w_\lam(x):=u(x^\lam)-u(x)
\]
(notice that $w_\lam$ is anti-symmetric),
 and 
\begin{equation}\label{eq.lam0}
\lam_0:=\sup\{\lam\leq0:w_\mu\geq 0\text{ in }\Sigma_\mu\text{ for any }\mu\leq \lam\}.
\end{equation}
This notation will be used throughout the paper.


The first step for the moving planes technique is to provide for a starting point to move the plane: for $\lam$ sufficiently negative, it must be showed that $w_\lam(x)\geq 0$ in $\Sigma_\lam$. This can be ensured by using the following maximum principle for anti-symmetric functions in bounded domains:
\begin{thm}[Maximum principle on bounded domains in hyperplanes] \label{maxpplehypbdd}


Let $T_\lambda,\:\Sigma_\lambda,w_\lambda$ be defined as above, $u\in C^{1,1}_{\text{loc}}(\R^n)\cap L_g$.

If 
\[
\left\{
\begin{array}{cc}
\gfls u (x^\lam)-\gfls u(x)\geq 0 & \text{ in }\Omega \\
w_\lam\geq 0 & \text{ in }\Sigma_\lam\setminus\Omega
\end{array}
\right.
\]
then
\[
w_\lam(x)\geq 0\text{ in }\Sigma_\lam.
\] 
Moreover, if $w_\lam=0$ at some point in $\Omega$, then $w_\lam\equiv 0$ in $\R^n$.

Moreover, the result holds true for unbounded domains if we further assume that $w_\lam(x)\geq 0$ as $|x|\to\infty$.
\end{thm}

The second step consists in proving that $\lam_0=0$ which, applying the result to $w_{\lambda_0}$ and $-w_{\lambda_0}$, implies that $u$ is symmetric about the plane $\{x_n=0\}$. This can be proved by means of a contradiction argument: by assuming that $\lam_0<0$ we can construct a sequence $\lam_j \searrow \lam_0$, and $x_j\in \Sigma_{\lam_j}$ such that 
\[
w_{\lam_j}(x_j)=\min_{\Sigma_{\lam_j}} w_{\lam_j}\leq 0. 
\]  
We will show that such a sequence contradicts the following boundary estimate:


\begin{prop}[Boundary estimate] \label{bdryest}
Let $\lam_0$ be given by \eqref{eq.lam0} and assume it is finite and that $w_{\lam_0}>0$ in $\Sigma_{\lambda_0}$. Suppose that there exists a sequence $\lam_j\searrow\lam_0$ and $x_j\in\Sigma_{\lam_j}$ such that 
\[
w_{\lam_j}(x_j)=\min_{\Sigma_{\lam_j}} w_{\lam_j}\leq 0\quad\text{ and }\quad\lim_{j\rightarrow\infty}x_j=\bar x\in T_{\lam_0}.
\]
Let $\delta_j=\mathrm{dist}(x_j,T_{\lam_j})$. Then
\[
\limsup_{j\rightarrow\infty}\frac{1}{\delta_j}\left(\gfls u(x^{\lambda_j}_j)-\gfls u(x_j)\right)< 0.
\]
\end{prop}
 
With the aid of the previous results, we can establish the symmetry of positive solutions under natural assumptions on the right hand side $f$; this is the content of the following two theorems, concerning bounded domains and the whole space, respectively.

\begin{thm}[Symmetry for solutions in a ball] \label{symmball}
Let $B$ be the unit ball in $\R^n$ and $u\in C^{1,1}_{\text{loc}}(B)\cap C(\overline{B})$ be a positive function in $B$ satisfying
\begin{align}\label{eqball}
\begin{cases}
\gfls u= f(u) & \text{ in }B \\
 u= 0 & \text{ in }\R^n\setminus B,
\end{cases}
\end{align}
where $f$ is a Lipschitz function with $f'$ nondecreasing and satisfying the following growth condition:
\begin{equation}\label{grwothf}
g'(t)\leq Cf'(t) \text{ for }0<t<1\text{ and some }C>0.
\end{equation}

Then $u$ is radially symmetric and monotone nondecreasing around the origin. 	
\end{thm}
 
%

\begin{thm}[Symmetry for solutions in whole space: decreasing RHS] \label{symmwhole1}
Let $u\in C^{1,1}_{\text{loc}}(\R^n)\cap L_g$ satisfy
\begin{equation} \label{eqrn}
\gfls u =f(u)\quad\text{ and }\quad u>0\quad\text{ in }\R^n.
\end{equation}
Assume 
\begin{equation}\label{cond-f}
f'(t)\leq 0 \quad \text{ for } t\leq 1, 
\end{equation}
\begin{equation}\label{cond-decay}
\lim_{|x|\to\infty} u(x)=0.
\end{equation} 

Then $u$ is radially symmetric around some point in $\R^n$. 	
\end{thm}

We leave as an open question to find which are the (best) conditions on $f$ and on the decay of $u$ at infinity in order to ensure symmetry of positive solutions $u\in C^{1,1}_{\text{loc}}(\R^n)\cap L_g$ of
\[
\gfls u =f(u)\quad\text{ and }\quad u>0\quad\text{ in }\R^n
\]
in the case in which $f$ is an increasing function.

Further interesting research directions would be to address qualitative properties of solutions unbounded domains, for instance
\[
\gfls u =f(u)\text{ in }\{x_n>0\}\quad\text{ and }\quad u=0\text{ on }\{x_n=0\},
\]
or more general unbounded domains such as those given by the epigraph of a Lipschitz function.

%
%
%
%

\subsection*{Organization of the paper}
This article is organized as follows. Section \ref{sec.prel} is devoted to introduce the notion of Young function and the proof of some useful inequalities, and several properties that the fractional $g-$Laplacian fulfills. In section \ref{sec.max} we prove the maximum principles in domains and hyperplanes as well as the Liouville theorem, namely, Theorems \ref{teo0}, \ref{maxpplehyp} and \ref{liou}. Section \ref{sec.maxandbdry} contains the proofs of the maximum principle on bounded domains in hyperplanes, i.e., Theorem \ref{maxpplehypbdd}, and the boundary estimate stated in Proposition \ref{bdryest}. In Section \ref{sec.symm} we deliver the proof of our symmetry results, namely Theorems \ref{symmball} and \ref{symmwhole1}. Finally, in section \ref{sec.ext} we introduce some applications and  extensions of our results.

\section{Preliminaries}\label{sec.prel}

In this section we give some preliminary definitions and technical results that will be used throughout the paper. We recall the  notion of Young function and present some simple technical inequalities that will be helpful. Then, we define the fractional $g-$Laplacian and prove some important properties of it, both useful for the rest of the paper and of independent interest.

\subsection{Young functions}\label{ssec.yfun}

An application $G\colon[0,\infty)\longrightarrow [0,\infty)$ is said to be a  \emph{Young function} if it admits the integral representation 
\[
G(t)=\int_0^t g(\tau)\,d\tau,
\] 
where the right-continuous function $g$ defined on $[0,\infty)$ has the following properties:
\begin{align*}
&g(0)=0, \quad g(t)>0 \text{ for } t>0, \\
&g \text{ is nondecreasing on } (0,\infty) \\
&\lim_{t\to\infty}g(t)=\infty.
\end{align*}

From these properties it is easy to see that a Young function $G$ is continuous, nonnegative, strictly increasing and convex on $[0,\infty)$. Further, we recall that we may extend $g$ to the whole $\R$ in an odd fashion: for $t<0$ $g(t)=-g(-t)$.
 
We will consider the class of Young functions such that $g=G'$ is an absolutely continuous function that satisfies the condition 
\begin{equation}\label{ellip} 
1<p^--1 \leq \frac{tg'(t)}{g(t)}\leq p^+-1<\infty,\quad t>0.
\end{equation}
This condition was first considered in the seminal work of G. Lieberman \cite{lieberman1991natural} and is the analogous to the ellipticity condition in the linear theory as it will be apparent later on; it essentially says that \eqref{ellip} means that $g(t)$ is ``trapped between powers''. Moreover, integrating \eqref{ellip} we have that $G$ verifies
\begin{equation}\label{eq.p}
2<p^-\leq  \frac{tg(t)}{G(t)} \leq p^+<\infty,\quad t>0.
\end{equation}


In \cite[Theorem 4.1]{KrRu61} it is shown that the upper bound in \eqref{ellip} (or in \eqref{eq.p}) is equivalent to the so-called \emph{$\Delta_2$ condition (or doubling condition)}, namely
\begin{equation} \tag{$\Delta_2$}
g(2t) \leq 2^{p^+-1} g(t), \qquad G(2t)\leq 2^{p^+} G(t) \qquad t\geq 0.
\end{equation}
It is easy to verify that this condition implies the existence of constants $C_1,C_2>0$ such that for any $a,b\geq0$
\begin{equation}\label{delta2}
g(a+b) \leq C_1(g(a)+g(b)), \qquad G(a+b) \leq C_2(G(a)+G(b)).
\end{equation}

%

Further, the inequalities 
\begin{equation}\label{minmax1} 
\min\{\alpha^{p^--1}, \alpha^{p^+-1}\} g(t) \leq g(\alpha t) \leq \max\{\alpha^{p^--1}, \alpha^{p^+-1}\}g(t)
\end{equation}
and
\begin{equation}\label{minmax2} 
\frac{\min\{\alpha^{p^-}, \alpha^{p^+}\}}{p^+}G(t) \leq G(\alpha t) \leq p^+\max\{\alpha^{p^-}, \alpha^{p^+}\}G(t)
\end{equation}
hold for any, $\alpha,t\geq 0$ (see \cite[Lemma2.1]{bonder2018ah}).

A final condition to be imposed on $g$ is that its derivative $g'$ (that exists a.e.) is nondecreasing; we point out that this is analogous of dealing with the degenerate case $p\ge 2$ for the fractional $p-$Laplacian.

\subsection{Some useful inequalities}

We include here some technical inequalities that will be used throughout the paper; the proofs are simple but included for the sake of completeness.

\begin{lem}\label{lem.2}  
Let $G$ be a Young function such that $g=G'$ satisfies \eqref{ellip}. Then there exists $C=C(p^+,p^-)>0$ such that 
$$
g(b)-g(a)\geq C g(b-a).
$$
for all $b\geq a$.
\end{lem} 
 
\begin{proof}
We split the proof in several cases.

\textit{Case 1:} $b\geq a \geq 0$. 

If $a\geq \frac{b}{2}$, then, for some $\xi \in (a,b)$ and using \eqref{ellip},
\begin{align*}
g(b)-g(a) &= g'(\xi)(b-a) \\
		  &\geq g'\left(\frac{b}{2}\right)(b-a)\\
		  &\geq g'\left(\frac{b-a}{2}\right)(b-a)\\
		  &\geq 2(p^--1) g\left(\frac{b-a}{2}\right) \\
		  &\geq 2^{2-p^+}(p^--1) g\left(b-a\right)
\end{align*}
where we have used \eqref{minmax1} for the last inequality. 

If $0<a<\tfrac{b}{2}$, then we use \eqref{minmax1} again and the fact that $g$ is nondecreasing:
\begin{align*}
g(b)-g(a)&\geq g(b)-g\left(\frac{b}{2}\right) \\
		 & \geq   \left(1-2^{1-p^+}\right)g(b)  \\
		&\geq \left(1-2^{1-p^+}\right)g(b-a).
\end{align*}

\textit{Case 2:} $a\leq b\leq 0$.

In this case $|a|\geq |b|\geq 0$ and we can use oddity and derive the inequality from the previous case:
$$
g(b)-g(a)=g(|a|)-g(|b|)\geq Cg(|a|-|b|)= Cg(b-a).
$$

\textit{Case 3:} $a\leq 0 \leq b$. 

Here $a$ and $b$ have different signs, then since $g$ is odd
$$
g(b)-g(a)=g(|b|)+g(|a|) \geq C g(|b|+|a|) = Cg(b-a),
$$
where we have used \eqref{delta2} for $g$. The proof is now completed.
\end{proof}

\begin{lem} \label{lemita}
For any $a,b \in \R$ and $g$ an absolutely continuous function such that $g'$ is nondecreasing, it holds that	
$$
|g(a+b)-g(a)| \leq |b| g'(|a|+|b|).
$$
\end{lem}

\begin{proof}
A straightforward computation gives (recall that $g'$ is nondecreasing)
\begin{align*}
|g(a+b)-g(a)| &= \left|b \int_0^1 g'(a+tb)\,dt \right|\\
&\leq |b| \int_0^1g'(|a|+t|b|)\,dt\\
&\leq |b| g'(|a|+|b|),
\end{align*}
from where the lemma follows.
\end{proof}

\begin{lem}\label{desig}  
Let $G$ be a Young function such that $G'=g$ satisfies \eqref{ellip}. There exists a constant $0<C_0=C_0(p^-,p^+)\leq 1$ such that, if we write
\[
g(b)-g(a)=g'(\xi)(b-a),
\]
then
\begin{equation} \label{relac}
|\xi|\geq C_0\max\{|b|,|a|\}.
\end{equation}
\end{lem} 
\begin{proof}
Without loss of generality we may assume that $|b|>|a|$. 

\textit{Case 1:} $|a|\geq \frac{|b|}{2}$. 

If $a$ and $b$ have the same sign, then $\xi$ is between $a$ and $b$ and \eqref{relac} holds with $C_0=1$. If $a$ and $b$ are of opposite signs, since $g$ is odd, $g(a)$ and $g(b)$ are of opposite signs. It follows that
$$
2g'(|\xi|)|b| \geq |g'(\xi)||b-a|=|g(b)-g(a)|\geq |g(b)|.
$$
Then, by using \eqref{ellip} and the fact that $|g(b)|=|g(|b|)|$,
$$
2g'(|\xi|) \geq \frac{|g(b)|}{|b|}    \geq \frac{g'(|b|)}{p^+-1}
$$
and the desired relation holds since $g'$ is nondecreasing.

\textit{Case 2:} $|a|\leq \frac{|b|}{2}$. 

In this case \eqref{ellip}  implies that
$$
g(|b|)\geq g(2|a|) \geq 2^{p^--1}g(|a|)
$$
from where
$$
2g'(|\xi|)|b|\geq |g(b)-g(a)|\geq |g(|b|)|-|g(|a|)|\geq (1-2^{-(p^--1)})g(|b|) 
$$
and \eqref{relac} follows.
\end{proof}

\subsection{The fractional $g-$Laplacian}

Recall the fractional $g-$Laplacian defined in \cite{FBS} by:
\begin{equation}\label{eq.gfls}
\gfls u(x):=\textrm{p.v.}\int_{\R^n}g\left(D_su(x,y)\right)\frac{dy}{|x-y|^{n+s}}
\end{equation}
where the notation
\[
D_su(x,y):=\frac{u(x)-u(y)}{|x-y|^s}
\]
is (and will be) used. 

In this section we include some elementary properties of the fractional $g-$Laplacian. $\gfls$ is an operator ``of order $2s$'', so to ensure that $\gfls$ is well defined at $x$ we need $u\in C^{2s+\delta}$ at $x$ for some $\delta>0$; for the purposes of this paper it will be enough to assume that $u\in C^{1,1}$ at $x$ in the sense that there exists $C\geq 0$ such that
\begin{equation}\label{c11}
|u(x)-u(y)-\nabla u(x)\cdot (y-x)|\leq C|y-x|^2\text{ for } |y-x|\text{ small enough.}
\end{equation}
This is a stronger regularity assumption than $C^{2s+\delta}$ but serves for the sake of clarity. 

On the other hand, because of the nonlocal nature of $\gfls$ we need to control the behavior of $u$ at infinity; we will denote 
\begin{equation}\label{integrability}
L_g:=\left\{u\in L^1_{\text{loc}}(\R^n) :\int_{\R^n}g\left(\frac{|u(x)|}{1+|x|^s}\right)\frac{dx}{1+|x|^{n+s}}<\infty\right\}.
\end{equation}

Notice that the inclusion 
\begin{equation}\label{inclusion}
L_g\subset L_{g'}
\end{equation}
holds; indeed, if $u\in L_g$ we can split
\[
\int_{\R^n}g'\left(\frac{|u(x)|}{1+|x|^s}\right)\frac{dx}{1+|x|^{n+s}}=\left(\int_{\left\{x:\frac{|u(x)|}{1+|x|^s}\leq  1\right\}}+\int_{\left\{x:\frac{|u(x)|}{1+|x|^s}>1\right\}} \right)
g'\left(\frac{|u(x)|}{1+|x|^s}\right)\frac{dx}{1+|x|^{n+s}}.
\]
Since $g'$ is nondecreasing the first term is bounded by
\[
g'(1)\int_{\R^n}\frac{dx}{1+|x|^{n+s}}<\infty
\]
whereas the second term is bounded (using the \eqref{ellip} and the fact that $\frac{|u(x)|}{1+|x|^s}>1$) by
\[
(p^+-1)\int_{\R^n}g\left(\frac{|u(x)|}{1+|x|^s}\right)\frac{dx}{1+|x|^{n+s}}<\infty.
\]

The next lemma shows that \eqref{c11} and \eqref{integrability} are enough for \eqref{eq.gfls} to be well defined.

\begin{lem} \label{puntual}
Let $u\in C^{1,1}\cap L_g$ at $x\in\R^n$, then $\gfls u(x)$ is well defined. 
\end{lem}

\begin{proof}
Let $0<\varepsilon<1$ and write 
\[
\int_{\R^n\setminus B_\varepsilon(x)}g\left(\frac{u(x)-u(y)}{|x-y|^s}\right)\frac{dy}{|x-y|^{n+s}}=I_1+I_2
\]
with
\[
I_1:=\int_{B_1(x)\setminus B_\varepsilon(x)}g\left(\frac{u(x)-u(y)}{|x-y|^s}\right)\frac{dy}{|x-y|^{n+s}}
\]
and
\[	
I_2:=\int_{\R^n\setminus B_1(x)}g\left(\frac{u(x)-u(y)}{|x-y|^s}\right)\frac{dy}{|x-y|^{n+s}}.
\]

On one hand we have
\[
u(x)-u(y)=\nabla u(x)\cdot(x-y)+O(|x-y|^{2})
\]
as $y\rightarrow x$ so Lemma \ref{lemita} gives (recall $|x-y|<1$)
\[
\left|g\left(\frac{\nabla u(x)\cdot(x-y)+O(|x-y|^2)}{|x-y|^s}\right)-g\left(\frac{\nabla u(x)\cdot(x-y)}{|x-y|^s}\right)\right| \leq C|x-y|^{2-s} g'\left(C|x-y|^{1-s}\right) 
\]

Next notice that $g\left(\frac{\nabla u(x)\cdot(x-y)}{|x-y|^s}\right)$ is odd so its integral over $B_1(x)\setminus B_\varepsilon(x)$ vanishes, whence 
\begin{align*}
\left|\int_{B_1(x)\setminus B_\varepsilon(x)}g\left(\frac{u(x)-u(y)}{|x-y|^s}\right)\frac{dy}{|x-y|^{n+s}}\right| & = \left|\int_{B_1(x)\setminus B_\varepsilon(x)}  \Big[g\left(\frac{\nabla u(x)\cdot(x-y)+O(|x-y|^2)}{|x-y|^s}\right)\right. \\ 
																												   & \left. -g\left(\frac{\nabla u(x)\cdot(x-y)}{|x-y|^s}\right)\frac{dy}{|x-y|^{n+s}}\Big]\right| \\	
																												   &\leq \int_{B_1(x)\setminus B_\varepsilon(x)}C|x-y|^{2-s} g'\left(C|x-y|^{1-s}\right)\frac{dy}{|x-y|^{n+s}} \\
																												   &\leq Cg'\left(C\right)\int_{B_1(x)\setminus B_\varepsilon(x)}\frac{dy}{|x-y|^{n-2(1-s)}}
\end{align*}
so $I_1$ converges as $\varepsilon\rightarrow0^+$.

On the other hand, by \eqref{delta2}
\begin{align*}
\int_{\R^n\setminus B_1(x)}g\left(\frac{u(x)-u(y)}{|x-y|^s}\right)\frac{dy}{|x-y|^{n+s}}& \leq C\left(\int_{\R^n\setminus B_1(x)}g\left(\frac{|u(x)|}{|x-y|^s}\right)\frac{dy}{|x-y|^{n+s}} \right.\\ 
& \left.+\int_{\R^n\setminus B_1(x)}g\left(\frac{|u(y)|}{|x-y|^s}\right)\frac{dy}{|x-y|^{n+s}}\right)										 
\end{align*}
and noticing that for $y$ outside $B_1(x)$ the polynomials $|x-y|^s$ and $1+|y|^s$ and $|x-y|^{n+s}$ and $1+|y|^{n+s}$ are comparable we get
\[
|I_2|\leq C\left(g(|u(x)|)\int_{\R^n\setminus B_1(x)}\frac{dy}{|x-y|^{n+s}}+\int_{\R^n\setminus B_1(x)}g\left(\frac{|u(y)|}{1+|y|^s}\right)\frac{dy}{1+|y|^{n+s}}\right).
\]
The first term is obviously integrable and so is the second owing to \eqref{integrability}. The result follows.
\end{proof}


The following simple result shows that $\gfls$ is rotation invariant:
\begin{lem} \label{rotation}
Let $u\in C^{1,1}\cap L_g$ at $x\in\R^n$ and let $Q\in\R^{n\times n}$ be an orthogonal matrix. Define
\[
u_Q(x):=u(Qx).
\]
Then $\gfls u_Q(x)=\gfls u(Qx)$.
\end{lem}

\begin{proof}
The proof is an immediate change of variables:
\begin{align*}
\gfls u_Q(x) & = \lim_{\ve\to 0}\int_{\R^n\setminus B_\varepsilon(x)}g\left(\frac{u(Qx)-u(Qy)}{|x-y|^s}\right)\frac{dy}{|x-y|^{n+s}} \\
			 & =\lim_{\ve\to 0}\int_{\R^n\setminus B_\varepsilon(x)}g\left(\frac{u(Qx)-u(Qy)}{|Qx-Qy|^s}\right)\frac{dy}{|Qx-Qy|^{n+s}} \\
			 & =\lim_{\ve\to 0}\int_{\R^n\setminus B_\varepsilon(Qx)}g\left(\frac{u(Qx)-u(z)}{|Qx-z|^s}\right)\frac{dz}{|Qx-z|^{n+s}}\\
			 & =\gfls u(Qx).
\end{align*}
\end{proof}

We will further need the following result concerning the fractional $g-$Laplacian of a cut-off function:
\begin{lem}\label{lem.1}
Let $\varphi\in C^\infty_c(B_1)$ be a radially symmetric function decreasing with $|x|$, then $\gfls\varphi(x)$ is well defined and
\begin{equation}\label{boundvarphi}
|\gfls\varphi(x)|\leq C
\end{equation}
for some $C$ depending of $n$, $s$ and $\|\varphi\|_{C^2(B_1)}$. Furthermore, $\gfls\varphi(x)$ is also a radial function. 
\end{lem}

\begin{proof}
The bound in \eqref{boundvarphi} follows simply by repeating the steps of the proof of Lemma \ref{puntual}. Let us show that $\gfls\varphi(x)$ is radial; making the change of variables to spherical coordinates and denoting $x=rx',\:y=\rho y'$ with $|x'|=|y'|=1$ we can compute
\begin{align*}
\gfls \varphi(x) & = \textrm{p.v.}\int_{\R^n}g\left(\frac{\varphi(x)-\varphi(y)}{|x-y|^s}\right)\frac{dy}{|x-y|^{n+s}} \\
			 	 & =\textrm{p.v.} \int_0^\infty\int_{\partial B_1}g\left(\frac{\varphi(r)-\varphi(\rho)}{|rx'-\rho y'|^s}\right)\frac{\rho^{n-1} dy'\:d\rho}{|rx'-\rho y'|^{n+s}} \\
				 & =\textrm{p.v.} \int_0^\infty\int_{\partial B_1}g\left(\frac{\varphi(r)-\varphi(\rho)}{r^s|x'-\frac{\rho}{r} y'|^s}\right)\frac{\rho^{n-1} dy'\:d\rho}{r^{n+s}|x'-\frac{\rho}{r} y'|^{n+s}} \\
	(\rho=r\tau)\quad & =r^{-n-s}\textrm{p.v.} \int_0^\infty\int_{\partial B_1}(r\tau)^{n-1}g\left(\frac{\varphi(r)-\varphi(r\tau)}{r^s|x'-\tau y'|^s}\right)\frac{ r dy'\:d\tau}{|x'-\tau y'|^{n+s}} \\
						& =r^{-s}\textrm{p.v.} \int_0^\infty\tau^{n-1}\left(\int_{\partial B_1}g\left(\frac{\varphi(r)-\varphi(r\tau)}{r^s|x'-\tau y'|^s}\right)\frac{ dy'}{|x'-\tau y'|^{n+s}}\right)d\tau \\
						& =r^{-s}\textrm{p.v.} \int_0^\infty\tau^{n-1}h(\tau)d\tau.
\end{align*}
It is not obvious a priori that $h$ is a function of $\tau$ alone (and not of $x'$), but if we let $z'\in\partial B_1$ and $Q$ be a orthogonal matrix such that $z'=Qx'$ by changing variables $y'=Qw'$ we have
\begin{align*}
\int_{\partial B_1}g\left(\frac{\varphi(r)-\varphi(r\tau)}{r^s|z'-\tau y'|^s}\right)\frac{ dy'}{|z'-\tau y'|^{n+s}} & =\int_{\partial B_1}g\left(\frac{\varphi(r)-\varphi(r\tau)}{r^s|Qx'-\tau Qw'|^s}\right)\frac{ dw'}{|Qx'-\tau Qw'|^{n+s}} \\
																													& =\int_{\partial B_1}g\left(\frac{\varphi(r)-\varphi(r\tau)}{r^s|x'-\tau w'|^s}\right)\frac{ dw'}{|x'-\tau w'|^{n+s}}.
\end{align*}
Therefore $h$ is indeed independent of $x'$ and $\gfls \varphi(x)$ depends only on $r$ as desired. 
\end{proof}

We end this section with a technical lemma that gives control of $\gfls u$ if we perturb it by a smooth function. 
\begin{lem}\label{lem.3}
Let $u\in C^{1,1}\cap L_g$ at $x$  and $\psi\in C_0^\infty(\R^n)$, then for all $\delta>0$ there exists $C_\delta>0$ such that 
\[
\left|\gfls(u+\varepsilon\psi)(x)-\gfls u(x)\right|\leq C_\delta\varepsilon+\omega(\delta)
\]
with $\omega$ a continuous function of $\delta$ satisfying $\omega(0)=0$.
\end{lem}


\begin{proof}
Denote $v_\ve(x):=u(x)+\ve\psi(x)$, then
\begin{align*}
\gfls v_\ve & (x)-\gfls u(x) =\\
&=\text{p.v.} \left(\int_{B_\delta^c(x)} + \int_{B_\delta(x)}  \right) \left[ g\left(\frac{v_\ve(x)-v_\ve(y)}{|x-y|^s} \right) - g\left(\frac{u(x)-u(y)}{|x-y|^s} \right)\right] \frac{dy}{|x-y|^{n+s}}\\
&:=I_1+I_2.
\end{align*}
Let us bound the first integral. Recall 
$$
g(b)-g(a)= (b-a)\int_0^1g'(a+t(b-a))\:dt=(b-a)g'(a+t_0(b-a))
$$
for some $t_0\in(0,1)$. We use this expression with $a=\frac{u(x)-u(y)}{|x-y|^s}$ and $b=\frac{v_\ve(x)-v_\ve(y)}{|x-y|^s}$ to get
$$
g(D_s v_\ve)- g(D_s u) = \ve\frac{\psi(x)-\psi(y)}{|x-y|^s}g'\left(\frac{u(x)-u(y)+t_0\ve(\psi(x)-\psi(y))}{|x-y|^s}\right).
$$
Then, 
\begin{align*}
|I_1| & \leq \int_{B_\delta^c(x)}\left|\frac{\ve(\psi(x)-\psi(y))}{|x-y|^s}g'\left(\frac{u(x)-u(y)+t_0\ve(\psi(x)-\psi(y))}{|x-y|^s}\right)\right|\frac{dy}{|x-y|^{n+s}} \\
	& \leq 2\ve\|\psi\|_\infty\int_{B_\delta^c(x)}g'\left(\frac{|u(x)|+|u(y)|+2\|\psi\|_\infty}{|x-y|^s}\right)\frac{dy}{|x-y|^{n+2s}} \\
	& \leq 2\ve C\|\psi\|_\infty\int_{B_\delta^c(x)}g'\left(\frac{|u(x)|+|u(y)|+2\|\psi\|_\infty}{1+|y|^s}\right)\frac{dy}{1+|y|^{n+2s}},
\end{align*}
where we have also used that $|x-y|^s\sim 1+|y|^s$ and $|x-y|^{n+2s}\sim 1+|y|^{n+2s}$ if $y\in B_\delta^c(x)$ (the constant $C$ depends of course on $\delta$). Proceeding as in the proof of \eqref{inclusion} and denoting 
\[
E:=\left\{y\in\R^n:\frac{|u(x)|+|u(y)|+2\|\psi\|_\infty}{1+|y|^s}\leq 1\right\}
\]
we further obtain
\[
|I_1|\leq 2\ve C\|\psi\|_\infty\left( \int_{ B_\delta^c(x)\cap E}\frac{g'(1) \,dy}{1+|y|^{n+2s}}+\int_{B_\delta^c(x)\cap E^c}g\left(\frac{|u(x)|+|u(y)|+2\|\psi\|_\infty}{1+|y|^s}\right)\frac{dy}{1+|y|^{n+2s}}\right)
\]
but
\begin{align*}
\int_{B_\delta^c(x)\cap E^c}g\left(\frac{|u(x)|+|u(y)|+2\|\psi\|_\infty}{1+|y|^s}\right)\frac{dy}{1+|y|^{n+2s}} & \leq C\left(\int_{B_\delta^c(x)\cap E^c}g\left(\frac{|u(y)|}{1+|y|^s}\right)\frac{dy}{1+|y|^{n+2s}} \right. \\ 
																												& \left.+g(|u(x)|+2\|\psi\|_\infty)\int_{B_\delta^c(x)\cap E^c}\frac{dy}{1+|y|^{n+2s}}\right).
\end{align*}
Given that the second integral is finite and so is the first (since $u\in L_g$) we obtain
\[
|I_1|\leq C_\delta\varepsilon.
\]

Now we turn to $I_2$: \eqref{c11} allows us to write
\[
v_\ve(x)-v_\ve(y)= \nabla v_\ve(x)\cdot (x-y)+O(|x-y|^2).
\]
Also, $\nabla v_\ve(x)\cdot (x-y)$ is anti-symmetric for $y\in B_\delta(x)$ and $g$ is an odd function we have
$$
\text{p.v.} \int_{B_\delta(x)} g\left(  \frac{\nabla v_\ve(x)\cdot (x-y)}{|x-y|^s}  \right)\frac{dy}{|x-y|^{n+s}}=0.
$$
Then, using Lemma \ref{lemita}
\begin{align*}
\left|\text{p.v.}\int_{B_\delta(x)} g\left( D_s v_\ve \right)  \frac{dy}{|x-y|^{n+s}} \right|&  = \left|\text{p.v.}\int_{B_\delta(x)} g\left( D_s v_\ve \right) \frac{dy}{|x-y|^{n+s}} \right. \\
																							&\left. -\text{p.v.} \int_{B_\delta(x)} g\left(  \frac{\nabla v_\ve(x)\cdot (x-y)}{|x-y|^s}  \right)\frac{dy}{|x-y|^{n+s}}\right| \\
																							&\leq  \text{p.v.} \int_{B_\delta(x)}O(|x-y|^{2-s}) g'\left(|\nabla v_\ve(x)||x-y|^{1-s} \right. \\
																							&\left.  +O(|x-y|^{2-s})\right)\frac{dy}{|x-y|^{n+s}}\\
																							&\leq  \text{p.v.}\, C_\delta \int_{B_\delta(x)}g'\left(|\nabla v_\ve(x)|\delta^{1-s}+\delta^{2-s})\right)\frac{dy}{|x-y|^{n+s-2}}\\
																							&\leq  \text{p.v.}\, C_\delta g'(C_\delta)\int_{B_\delta(x)}\frac{dy}{|x-y|^{n+s-2}}<\infty.
\end{align*}
The bound for $g\left( D_s u \right)$ is analogous and we get
\[
|I_2|\leq \text{p.v.} \, C_\delta g'(C_\delta )\int_{B_\delta(x)}\frac{dy}{|x-y|^{n+s-2}},
\]
and the conclusion of the lemma follows.
\end{proof}

\begin{rem}
It may be worth pointing out that the constant $C_\delta$ depends on other quantities besides from $\delta$; it depends on $n,s,\|\psi\|_\infty$ and on $u$ itself. However, for the purposes of our application in the proof of Theorem \ref{maxpplehyp}, the important property is that it does not depend on $\varepsilon$.
\end{rem}
\section{Proof of the maximum principle on domains, maximum principle on hyperplanes and Liouville theorem}\label{sec.max}

This section is devoted to the proof of Theorems \ref{teo0}, \ref{maxpplehyp} and \ref{liou}. The first one is rather simple owing to the nonlocal nature of the operator:

\begin{proof}[Proof of Theorem \ref{teo0}]
Suppose that the conclusion is false. Then, since $u$ is continuous in $\overline \Omega$, there exists $\bar x\in\Omega$ such that 
\[
u(\bar x)=\min_\Omega u<0.
\]

By Lemma \ref{puntual} we can evaluate point-wisely the operator, then the last claim together with the fact the $u\geq 0$ in $\R^n \setminus \Omega$ give that
\begin{align*}
\gfls u(\bar x) &=\textrm{p.v.} \left( \int_\Omega + \int_{\R^n\setminus \Omega} \right)	  g\left(\frac{u(\bar x)-u(y)}{|\bar x-y|^s}\right)\frac{dy}{|\bar x-y|^{n+s}}\\
&< 
\int_{\R^n\setminus \Omega} g\left(\frac{u(\bar x)-u(y)}{|\bar x-y|^s}\right)\frac{dy}{|\bar x-y|^{n+s}}\\
&\leq 0.
\end{align*} 
This contradicts that $\gfls u(\bar x)\geq 0$ by hypothesis, and hence we must have $u(x)\geq 0$ in $\Omega$.

If $u(\bar x)=0$ at some $\bar x\in\Omega$, then
$$
0\leq \gfls u(\bar x) = \textrm{p.v.}\int_{\R^n} g\left(\frac{-u(y)}{|\bar x-y|^s}\right)\frac{dy}{|\bar x-y|^{n+s}}\leq 0
$$
and hence the integral must be identically zero. Since $u$ is non-negative, we conclude that $u(x)\equiv 0$ in $\R^n$ and the proof concludes.
\end{proof}


Next we give the 
\begin{proof}[Proof of Theorem \ref{maxpplehyp}]

By the rotation and translation invariance of $\gfls$ we may assume that 
\[
H=\{x\in\R^n:x_1=0\}\quad \text{and}\quad \Sigma=\{x\in\R^n:x_1<0\}.
\]

By contradiction, let us suppose \eqref{eq1} is false and let 
$$
A:=\sup_\Sigma w(x)> 0.
$$
Then, for $\gamma\in(0,1)$ to be chosen later there exists $\bar x\in \Sigma$ such that 
\[
w(\bar x)\geq \gamma A.
\] 
Let $\eta\in C^\infty_c(B_1)$ be a radially symmetric, decreasing function satisfying 
\[
0\leq \eta\leq 1,\quad \eta(0)=1.
\] 
Recall that, due to Lemma \ref{lem.1},
\[
|\gfls\eta|\leq C
\]
and $\gfls\eta$ is a radial function. Let us further set 
\[
\psi(x):=\eta(x+\bar x)\quad\text{and}\quad\tpsi(x):=\eta(x-\bar x)
\]
and notice that 
\begin{align*}
& \tpsi(x)-\psi(x)\text{ is antisymmetric with respect to }H,\\
&\tpsi(x)=0\text{ in }\R^n\setminus B_1(\bar x), \\
& \psi(x)=0\text{ in }\R^n\setminus B_1(-\bar x)\:(\text{in particular in } \Sigma). 
\end{align*} 

The idea is to construct an appropriate perturbation of $w$ and use \eqref{eq.ineq} to get a contradiction. We start by choosing $\varepsilon>0$ small enough so that 
\[
w(\bar x)+\varepsilon\tpsi(\bar x)-\varepsilon\psi(\bar x)=w(\bar x)+\varepsilon\geq A
\]
and notice that by construction 
\[
w(\bar x)+\varepsilon\tpsi(\bar x)-\varepsilon\psi(\bar x)\geq w(x)+\varepsilon\tpsi(x)-\varepsilon\psi(x)\quad\forall x\in \Sigma\setminus B_1(\bar x)
\]
and therefore
\[
\max_{x\in \Sigma} (w(x)+\varepsilon\tpsi(x)-\varepsilon\psi(x))=w(\bar{x})+\varepsilon\tpsi(\bar{x})-\varepsilon\psi(\bar{x})\text{ for some }\bar{x}\in \overline{B_1(\bar x)}.
\]

We will estimate 
\begin{equation}\label{eq.diff}
\gfls(\ut+\varepsilon\tpsi)(\bar{x})-\gfls(u+\varepsilon\psi)(\bar{x})
\end{equation}
by above and below to reach a contradiction. 

We start computing
\begin{align*}
&\gfls(\ut+\varepsilon\tpsi)(\bar{x})-\gfls(u+\varepsilon\psi)(\bar{x})= \\
&= \int_{\R^n} \left[g\left(\frac{\ut(\bar{x})+\varepsilon\tpsi(\bar{x})-\ut(y)-\varepsilon\tpsi(y)}{|\bar{x}-y|^s}\right)-g\left(\frac{u(\bar{x})+\varepsilon\psi(\bar{x})-u(y)-\varepsilon\psi(y)}{|\bar{x}-y|^s}\right)\right]\frac{dy}{|\bar{x}-y|^{n+s}}.
\end{align*}
Splitting $\R^n$ as $\Sigma\cup\Sigma^c$, and performing a change of variables the expression above reads as
\begin{align*}
&\int_\Sigma \left[g\left(\frac{\ut(\bar{x})+\varepsilon\tpsi(\bar{x})-\ut(y)-\varepsilon\tpsi(y)}{|\bar{x}-y|^s}\right)-g\left(\frac{u(\bar{x})+\varepsilon\psi(\bar{x})-u(y)-\varepsilon\psi(y)}{|\bar{x}-y|^s}\right)\right]\frac{dy}{|\bar{x}-y|^{n+s}} \\ 
& +\int_\Sigma \left[g\left(\frac{\ut(\bar{x})+\varepsilon\tpsi(\bar{x})-u(y)-\varepsilon\psi(y)}{|\bar{x}+y|^s}\right)-g\left(\frac{u(\bar{x})+\varepsilon\psi(\bar{x})-\ut(y)-\varepsilon\tpsi(y)}{|\bar{x}+y|^s}\right)\right]\frac{dy}{|\bar{x}+y|^{n+s}}
\end{align*}
where we have used the definition of $\ut$ and $\tilde \psi$. We can further rewrite this as	
\begin{align*}
&\int_\Sigma \left(\frac{1}{|\bar{x}-y|^{n+s}}-\frac{1}{|\bar{x}+y|^{n+s}}\right) 
\left[  g\left(\frac{\ut(\bar{x})+\varepsilon\tpsi(\bar{x})-\ut(y)-\varepsilon\tpsi(y)}{|\bar{x}-y|^s}\right) \right. \\ 
& \hspace{5.5cm} \left. -g\left(\frac{u(\bar{x})+\varepsilon\psi(\bar{x})-u(y)-\varepsilon\psi(y)}{|\bar{x}-y|^s}\right)\right] \:dy\\
& +\int_\Sigma \left[g\left(\frac{\ut(\bar{x})+\varepsilon\tpsi(\bar{x})-\ut(y)-\varepsilon\tpsi(y)}{|\bar{x}-y|^s}\right) -g\left(\frac{u(\bar{x})+\varepsilon\psi(\bar{x})-\ut(y)-\varepsilon\tpsi(y)}{|\bar{x}+y|^s}\right)\right.\\
&\quad +\left.g\left(\frac{\ut(\bar{x})+\varepsilon\tpsi(\bar{x})-u(y)-\varepsilon\psi(y)}{|\bar{x}+y|^s}\right)-g\left(\frac{u(\bar{x})+\varepsilon\psi(\bar{x})-u(y)-\varepsilon\psi(y)}{|\bar{x}-y|^s}\right)\right]\frac{dy}{|\bar{x}+y|^{n+s}} \\
& := I_1+I_2+I_3.
\end{align*}

We bound each term separately. Notice that 
\begin{equation} \label{maior}
\frac{1}{|\bar{x}-y|^{n+s}}\geq\frac{1}{|\bar{x}+y|^{n+s}} \quad\text{for }y\in \Sigma
\end{equation}
and observe that
\begin{align*}
&0\leq (w(\bar{x})+\varepsilon\tpsi(\bar{x})- \ve \psi(\bar{x}))-(w(y)+\varepsilon\tpsi(y)-\ve\psi(y)) \\ 
&= (\ut(\bar{x})+\varepsilon\tpsi(\bar{x})-\ut(y)-\varepsilon\tpsi(y))-(u(\bar{x})+\varepsilon\psi(\bar{x})-u(y)-\varepsilon\psi(y))
\end{align*}
so the monotonicity of $g$ implies that $I_1\geq0$.

To bound $I_2$, first observe that
\begin{align*}
(\ut(\bar{x})+\varepsilon\tpsi(\bar{x}) & -\ut(y)-\varepsilon\tpsi(y)) - (u(\bar{x})+\varepsilon\psi(\bar{x})-\ut(y)-\varepsilon\tpsi(y))\\
&= w(\bar x) + \ve \tpsi(\bar x) - \ve \psi(\bar x)\geq A,
\end{align*}
therefore, by \eqref{maior} and Lemma \ref{lem.2} we get
\[
g\left(\frac{\ut(\bar{x})+\varepsilon\tpsi(\bar{x})-\ut(y)-\varepsilon\tpsi(y)}{|\bar{x}-y|^s}\right) -g\left(\frac{u(\bar{x})+\varepsilon\psi(\bar{x})-\ut(y)-\varepsilon\tpsi(y)}{|\bar{x}+y|^s}\right)\geq Cg\left(\frac{A}{|\bar{x}-y|^s}\right).
\]
and therefore
\begin{align*}
I_2 & \geq  C\int_\Sigma g\left(\frac{A}{|\bar{x}-y|^s}\right)\frac{dy}{|\bar{x}+y|^{n+s}} \\
	& =C_0g(A).	
\end{align*}
Finally, notice that the bound for $I_3$ is similar to the one for $I_2$ and we get
\begin{equation}\label{eq.2}
\text{\eqref{eq.diff}}\geq 2C_0g(A).
\end{equation}

To get an upper bound, we use Lemma \ref{lem.3} and \eqref{eq.ineq}
\begin{align*}
\text{\eqref{eq.diff}} & = \gfls(\ut+\varepsilon\tpsi)(\bar{x})-\gfls\ut(\bar{x}) \\
					   & \quad -(\gfls(u+\varepsilon\psi)(\bar{x})-\gfls u(\bar{x})) + \gfls\ut(\bar{x})-\gfls u(\bar{x}) \\
					   & \leq 2(C_\delta\varepsilon+\omega(\delta)) + \gfls\ut(\bar x) -  \gfls u(\bar x)\\
					   & \leq  2(C_\delta\varepsilon+\omega(\delta)).
\end{align*}

The last inequality, together with \eqref{eq.2} gives
\[
C_0 g(A)\leq C_\delta\varepsilon+\omega(\delta)
\]
which, taking $\delta$ such that $\omega(\delta)\leq\frac{C_0 g(A)}{2}$ gives
\[
\frac{C_0g(A)}{2C_\delta}\leq \varepsilon,
\]
but since $\varepsilon$ can be chosen as small as needed we have reached a contradiction.
\end{proof}

As mentioned in the Introduction, the previous theorem is sufficient to give the 
\begin{proof}[Proof of Theorem \ref{liou}]

Let us see that $u$ is symmetric with respect to any hyperplane, as a consequence of Theorem \ref{maxpplehyp}. Indeed, if $H$ is any hyperplane and $\Sigma$ is the semi-space on one side of $H$, we can define $w$ as in Theorem \ref{maxpplehyp} and notice that since $u$ is bounded so is $w$. Further, since $u$ is $g-$harmonic in $\R^n$,
\[
\gfls \ut(x)-\gfls u(x) = 0,
\]
in particular  if  $\ut (x)>u(x)$ for some $x\in \Sigma$. 

By Theorem \ref{maxpplehyp} $w(x)\leq 0$ in $\Sigma$. Similarly it can be proved that $w(x)\geq 0$ in $\Sigma$. Therefore $w(x)\equiv 0$ in $\Sigma$, and $u(x)$ is symmetric with respect to $H$. Since $H$ can be chosen arbitrarily, $u$ is radially symmetric about any point, giving that $u(x)\equiv C$ and concluding the proof.
\end{proof}

\section{Proof of Theorem \ref{maxpplehypbdd} and Proposition \ref{bdryest}}\label{sec.maxandbdry}

This section is dedicated to the proofs of Theorem \ref{maxpplehypbdd} and Proposition \ref{bdryest}, starting with the former:

\begin{proof}[Proof of Theorems \ref{maxpplehypbdd}]

Assume that the thesis of the theorem fails to hold, that is for some $\bar x\in \Omega$
\[
w_\lam(\bar x)=\min_{\Sigma_\lam}w=\min_{\Omega}w<0.
\] 
Let us compute, splitting $\R^n$ as $\Sigma_\lam\cup\Sigma_\lam^c$, 
\begin{align*}
\gfls u(\bar x^\lam)-\gfls u(\bar x) & = \int_{\R^n}\left[g\left(\frac{u(\bar x^\lam)-u(y^\lam)}{|\bar x-y|^s}\right)-g\left(\frac{u(\bar x)-u(y)}{|\bar x-y|^s}\right)\right]\frac{dy}{|\bar x-y|^{n+s}} \\
& = \int_{\Sigma_\lam}\left[g\left(\frac{u(\bar x^\lam)-u(y^\lam)}{|\bar x-y|^s}\right)-g\left(\frac{u(\bar x)-u(y)}{|\bar x-y|^s}\right)\right]\frac{dy}{|\bar x-y|^{n+s}}\\
& + \int_{\Sigma_\lam}\left[g\left(\frac{u(\bar x^\lam)-u(y)}{|\bar x-y^\lam|^s}\right)-g\left(\frac{u(\bar x)-u(y^\lam)}{|\bar x-y^\lam|^s}\right)\right]\frac{dy}{|\bar x-y^\lam|^{n+s}}\\
& = I_1+I_2
\end{align*}
with
\[
I_1:=\int_{\Sigma_\lam}\left[g\left(\frac{u(\bar x^\lam)-u(y^\lam)}{|\bar x-y|^s}\right)-g\left(\frac{u(\bar x)-u(y)}{|\bar x-y|^s}\right)\right]\left(\frac{1}{|\bar x-y|^{n+s}}-\frac{1}{|\bar x-y^\lambda|^{n+s}}\right)\:dy
\]
and
\begin{align*}
I_2 & := \int_{\Sigma_\lam}\left[g\left(\frac{u(\bar x^\lam)-u(y)}{|\bar x-y^\lam|^s}\right)-g\left(\frac{u(\bar x)-u(y^\lam)}{|\bar x-y^\lam|^s}\right) \right.\\	
	&+\left.g\left(\frac{u(\bar x^\lam)-u(y^\lam)}{|\bar x-y|^s}\right)-g\left(\frac{u(\bar x)-u(y)}{|\bar x-y|^s}\right)\right]\frac{dy}{|\bar x-y^\lam|^{n+s}}.
\end{align*}

Now, 
\[
\frac{1}{|\bar x-y|^{n+s}}-\frac{1}{|\bar x-y^\lambda|^{n+s}}>0
\]
while, recalling that  
\[
w_\lambda(\bar x)-w_\lambda(y)= u(\bar x^\lam)-u(y^\lam)-(u(\bar x)-u(y))\leq 0,
\]
we get
\[
g\left(\frac{u(\bar x^\lam)-u(y^\lam)}{|\bar x-y|^s}\right)-g\left(\frac{u(\bar x)-u(y)}{|\bar x-y|^s}\right)\leq 0
\]
so that 
\begin{equation}\label{eq.i1}
I_1\leq 0.
\end{equation}
For the other term we have
\[
I_2=w_\lambda(\bar x)\int_{\Sigma_\lam}\frac{g'(\xi(y))+g'(\zeta(y))}{|\bar x-y^\lambda|^{n+2s}}\:dy
\]
with 
\begin{align*}
\xi(y)\text{ between }  u(\bar x^\lam)&-u(y)\text{ and } u(\bar x)-u(y)\\
&\text{and}\\
\zeta(y)\text{ between }  u(\bar x^\lam)&-u(y^\lam)\text{ and }u(\bar x)-u(y^\lam).
\end{align*}
Since by the contradiction assumption $w_\lambda(\bar x)<0$ and $G$ is convex (which implies $g'>0$) we have 
\begin{equation}\label{eq.i2}
I_2<0.
\end{equation}

Putting together \eqref{eq.i1} and \eqref{eq.i2} we get
\begin{equation} \label{es.negativo}
\gfls u(\bar x^\lam)-\gfls u(\bar x) <0
\end{equation}
contradicting the hypothesis.

On the other hand, if $w_\lam(x)=0$ at some $x\in\Omega$, then $x$ is a minimum of $w$ in $\Omega$. Therefore, by using the hypothesis and splitting the integrals as before,
$$
0\leq \gfls u(\bar x^\lam)-\gfls u(\bar x) = I_1+I_2
$$
with $I_2=0$, and then $I_1\geq 0$. This implies that 
$$
g\left(\frac{u(\bar x^\lam)-u(y^\lam)}{|\bar x-y|^s}\right)-g\left(\frac{u(\bar x)-u(y)}{|\bar x-y|^s}\right)\geq 0.
$$
Observe that the monotonicity of $g$ implies that
$$
\left(g\left(\frac{u(\bar x^\lam)-u(y^\lam)}{|\bar x-y|^s}\right)-g\left(\frac{u(\bar x)-u(y)}{|\bar x-y|^s}\right)\right) \frac{(u(\bar x^\lam) - u(\bar x)) - (u(y)-u(y^\lam))}{|x-y|^s} \geq 0
$$
from where it is derived that
$$
(u(\bar x^\lam) - u(\bar x)) - (u(y)-u(y^\lam)) = w_\lam(x)-w_\lam(y) = -w_\lam(y)\geq 0,
$$
giving that $w_\lam(y)=0$ in $\Sigma_\lam$, and from the antisymmetry of $w_\lam$, in $\R^n$.

Finally, when $\Omega$ is unbounded, if we further assume that $w(x)\geq 0$ as $|x|\to \infty$, if it assumed that $w(x)\geq 0$ in $\Sigma_\lam$ does not hold,  a similar reasoning can be performed to reach a contradiction.
\end{proof}

%

Recall that, to prove Theorem \ref{symmball} we need Proposition \ref{bdryest} so that is the next proof we address.
 
\begin{proof}[Proof of Proposition \ref{bdryest}]
Proceeding as in the previous proof we compute
\begin{align*}
&\frac{1}{\delta_j}\left(\gfls u(x_j^{\lambda_j})-\gfls u(x_j) \right) =  \\
& =\frac{1}{\delta_j}\int_{\Sigma_{\lambda_j}}\left[g\left(\frac{u(x_j^{\lambda_j})-u(y^{\lambda_j})}{|x_j-y|^s}\right)- g\left(\frac{u(x_j)-u(y)}{|x_j-y|^s}\right)\right]\left(\frac{1}{|x_j-y|^{n+s}}-\frac{1}{|x_j-y^{\lambda_j}|^{n+s}}\right)\:dy \\
& +\frac{w_{\lambda_j}(x_j)}{\delta_j}\int_{\Sigma_{\lambda_j}}\frac{g'(\xi(y))+g'(\zeta(y))}{|x_j-y^{\lambda_j}|^{n+2s}}\:dy.
\end{align*}
Recall that
\begin{equation}\label{bdry1}
\frac{w_{\lambda_j}(x_j)}{\delta_j}\int_{\Sigma_{\lambda_j}}\frac{g'(\xi(y))+g'(\zeta(y))}{|x_j-y^{\lambda_j}|^{n+2s}}\:dy\leq 0
\end{equation} 
since $w_{\lambda_j}(x_j)\leq0$ and $g'>0$. 

For the first term, we first note that
\[
\lim_{j\rightarrow\infty}u(x_j^{\lambda_j})-u(y^{\lambda_j})-u(x_j)-u(y)=\lim_{j\rightarrow\infty}w_{\lambda_j}(x_j)-w_{\lambda_j}(y)=w_{\lambda_0}(\bar x)-w_{\lambda_0}(y)<0   
\]
and therefore
\begin{equation}\label{bdry2}
\lim_{j\rightarrow\infty}g\left(\frac{u(x_j^{\lambda_j})-u(x_j)}{|\bar x-y|^s}\right)- g\left(\frac{u(x_j)-u(y)}{|\bar x-y|^s}\right)<0.
\end{equation}
Also, 
\[
\frac{1}{|x_j-y|^{n+s}}-\frac{1}{|x_j-y^{\lambda_j}|^{n+s}} = \frac{-(n+s)}{2|\eta(y)|^{n+s+2}}\left(|x_j-y|-|x_j-y^{\lambda_j}|\right)
\]
and so
\begin{equation}\label{bdry3}
\lim_{j\rightarrow\infty}\frac{1}{\delta_j}\left(\frac{1}{|x_j-y|^{n+s}}-\frac{1}{|x_j-y^{\lambda_j}|^{n+s}}\right)>0.
\end{equation}

Gathering \eqref{bdry2}-\eqref{bdry3} with \eqref{bdry1} and taking $\limsup$ gives
\[
\limsup_{j\rightarrow\infty}\frac{1}{\delta_j}\left(\gfls u(x_j^{\lambda_j})-\gfls u(x_j) \right)<0
\]
as desired.
\end{proof}

\section{Symmetry results}\label{sec.symm}

Next we give present the proofs of our symmetry results, Theorems \ref{symmball} and \ref{symmwhole1}:

\begin{proof}[Proof of Theorem \ref{symmball}]

As mentioned, the proof follows the scheme of the moving planes method. Let us set 
\[
\Omega_\lambda:=\Sigma_\lambda\cap B.
\]

The first step is to show that for $\lambda>-1$ but sufficiently close to $-1$ we have $w_\lambda\geq 0$ in $\Omega_\lambda$. Let us assume for the sake of contradiction that this is not the case. Then
\[
w_\lambda(\bar x)=\min_{\Omega_\lambda}w_\lambda<0
\]
($\lambda$ will be suitably chosen later). 

On one hand, \eqref{eqball} gives
\[
\gfls u(\bar x^\lam)-\gfls u(\bar x)= f(u(\bar x^\lam))-f(u(\bar x))=f'(\xi)w_\lambda(\bar x) 
\]
for some $\xi$ that lies between $u(\bar x^\lam)$ and $u(\bar x)$. Since $f'$ is nondecreasing and by hypothesis $w_\lambda(\bar x)<0$ this turns into
\begin{equation}\label{diferencia1}
\gfls u(\bar x^\lam)-\gfls u(\bar x) \geq f'(u(\bar x))w_\lambda(\bar x).
\end{equation}

On the other hand, we can proceed as in the proof of Theorem \ref{maxpplehypbdd} to get
\begin{equation*}
\gfls u(\bar x^\lam)-\gfls u(\bar x) \leq w_{\lambda}(\bar x)\int_{\Sigma_{\lambda}}\frac{g'(\xi(y))+g'(\zeta(y))}{|\bar x-y^{\lambda}|^{n+2s}}\:dy=:w_{\lambda}(\bar x)I
\end{equation*}
with 
\begin{align*}
u(\bar{x}^\lam)-u(y^\lam)<  \xi(y)  < u(\bar{x})-u(y^\lam) \\
u(\bar{x}^\lam)-u(y) <  \zeta(y)  <u(\bar{x})-u(y).
\end{align*}
We will use Lemma \ref{desig} to bound $I$ by below:
\begin{align*}
I & \geq \int_{\Sigma_{\lambda}}\frac{ g'(C_0\max\{|u(\bar{x}^\lam)-u(y^\lam)|,|u(\bar{x})-u(y^\lam)|\})+g'(C_0\max\{|u(\bar{x}^\lam)-u(y)|,|u(\bar{x})-u(y)|\})}{|\bar x-y^{\lambda}|^{n+2s}}\:dy \\
&\geq \int_{\Sigma_\lam\setminus\Omega_\lam} \frac{g'(C_0\max\{|u(\bar{x}^\lam)|,|u(\bar{x})|\})}{|\bar x-y^{\lambda}|^{n+2s}}\:dy\\
&\geq \int_{\Sigma_\lam\setminus\Omega_\lam} \frac{g'(C_0|u(\bar{x})|)}{|\bar x-y^{\lambda}|^{n+2s}}\:dy\\
&\geq \int_{\Omega_{\lam+1}\setminus\Omega_\lam} \frac{g'(C_0|u(\bar{x})|)}{|\bar x-y^{\lambda}|^{n+2s}}\:dy\\
&\geq  \frac{g'(C_0|u(\bar{x})|)}{|1+\lam|^{2s}}
\end{align*}
where we have used that $u(y)=0$ in $\Sigma_\lam\setminus \Omega_\lam$ and $w_\lam(\bar x)=u(\bar x^\lam) - u(\bar x)<0$.

This together with \eqref{diferencia1} give 
\[
0\leq \left(\frac{g'(C_0u(\bar x))}{|\lambda+1|^{2s}}-f'(u(\bar x))\right) w_\lambda(\bar x)
\]
but thanks to \eqref{grwothf} (and the fact that $C_0\leq 1$) we know that   
$$
f'(u(\bar x))\geq f'(C_0 u(\bar x))\geq \frac{1}{C} g'(C_0u(\bar x))
$$
so that 
\[
0\leq \left(\frac{g'(C_0u(\bar x))}{|\lambda+1|^{2s}}-f'(u(\bar x))\right) w_\lambda(\bar x)\leq \left(\frac{1}{C|\lambda+1|^{2s}}-1\right) g'(C_0u(\bar x))w_\lambda(\bar x)
\]
and we can choose $\lambda$ sufficiently close to $-1$ so that   
\[
\left(\frac{1}{C|\lambda+1|^{2s}}-1\right)>0
\]
and we arrived to a contradiction.

Therefore, $w_\lambda\geq 0$ in $\Omega_\lambda$ for some $\lambda>-1$; next we want to show that $\lam_0=0$. Assume the contrary; then the Maximum Principle \ref{teo0} implies that 
\begin{equation}\label{wlam}
w_{\lambda_0}>0\quad\text{ in }\Omega_{\lambda_0}.
\end{equation}
By definition of supremum, the exists a sequence $\lambda_j$ such that 
\[
\lambda_j\leq \lambda_{j-1},\quad\lambda_j\leq 0,\quad \lim_{j\rightarrow\infty}\lambda_j=\lambda_0
\]
and 
\[
w_{\lambda_j}(x_j)=\min_{\Omega_{\lambda_j}}w_{\lambda_j}<0
\]
for some $x_j\in\Omega_{\lambda_j}$. We may assume further (up to taking a subsequence if needed) that 
\[
\lim_{j\rightarrow\infty}x_j=x_0\text{ and }w_{\lambda_0}(x_0)\leq 0
\]
which owing to \eqref{wlam} implies $x_0\in T_{\lam_0}$. 

Further, setting $\delta_j:=\textrm{dist}(x_j,T_{\lam_j})=|x_j-z_j|$ for some $z_j\in T_{\lambda_j}$ the equation gives
\[
\frac{1}{\delta_j}\left(\gfls u(x^{\lambda_j}_j)-\gfls u(x_j)\right)=f'(\xi_j)\frac{w_{\lambda_j}(x_j)}{\delta_j}=f'(\xi_j)\frac{w_{\lambda_j}(x_j)}{|x_j-z_j|}.
\]
Notice that $w_{\lambda_j}\equiv0$ on $T_{\lambda_j}$ so
\[
\lim_{j\rightarrow\infty}\frac{w_{\lambda_j}(x_j)}{|z_j-x_j|}=\lim_{j\rightarrow\infty}\frac{w_{\lambda_j}(x_j)-w_{\lambda_j}(z_j)}{|z_j-x_j|}=\lim_{j\rightarrow\infty}\frac{\nabla w_{\lambda_j}(z_j)\cdot(z_j-x_j)+o(|z_j-x_j|)}{|z_j-x_j|}=0
\]
which, together with the previous line implies 
\[
\lim_{j\rightarrow\infty}\frac{1}{\delta_j}\left(\gfls u(x^{\lambda_j}_j)-\gfls u(x_j)\right)=0
\]
(notice that $\xi_j$ is a bounded sequence and hence so is $f'(\xi_j)$).

But Proposition \ref{bdryest} gives
\[
\limsup_{j\rightarrow\infty}\frac{1}{\delta_j}\left(\gfls u(x^{\lambda_j}_j)-\gfls u(x_j)\right)<0,
\]
a contradiction.

Therefore, $\lam_0=0$; since we can choose the opposite direction to reflect we get the symmetry of $u$ about the $x_n$ direction. Finally, the rotation invariance of $\gfls$ implies that we can repeat the argument in any direction, so $u$ is symmetric about the origin. 
\end{proof}

\begin{proof}[Proof of Theorem \ref{symmwhole1}]
We split the proof in the two steps in order to apply the moving planes method.

\textit{Step 1.} Let us see that for $\lam$ sufficiently negative, $w_\lam(x)\geq 0$ for all $x\in \Sigma_\lam$.

Let us assume the opposite and obtain a contradiction. Due to the decay condition \eqref{cond-decay} on $u$,  there exists $\bar x\in\Sigma_\lam$ such that $w_\lam(\bar x)=\min_{\Sigma_\lam} w_\lam <0$.

Moreover,  from \eqref{eqrn} we get that
\begin{equation} \label{equacion}
\gfls u(x^\lambda)-\gfls u(x)= f(u(x^\lam))-f(u(x))=f'(\xi)w_\lambda(x) 
\end{equation}
where  $\xi$ lies between $u(x^\lam)$ and $u(x)$.  In particular, we have
$$
u(\bar x^\lam) \leq \xi(\bar x)\leq u(\bar x).
$$                                                 
Because of the decay assumption on $u$, for $\lam$ sufficiently negative, $u(\bar x)$ is small, and then $\xi(\bar x)$ is small, giving that $f'(\xi(\bar x))\leq 0$ due to \eqref{cond-f}. As a consequence
$$
\gfls u(x^\lam)-\gfls u(x)= f(u(x^\lam))-f(u(x))\geq 0. 
$$
However, as seen in \eqref{es.negativo}, under these conditions we have that
$$
\gfls u(x^\lam)-\gfls u(x)= f(u(x^\lam))-f(u(x))< 0,
$$
which is a contradiction. Therefore $w_\lam(x)\geq 0$ for all $x\in \Sigma_\lam$ for $\lam$ sufficiently negative.

\textit{Step 2}. Define $\lam_0 =\sup\{\lam\colon w_\mu(x)\geq 0, x\in \Sigma_\mu, \mu\leq \lam\}$.

Let us see that $u$ is symmetric about the limiting plane $T_{\lam_0}$, or
\begin{equation} \label{aprobar}
w_{\lam_0}(x)\equiv 0, \quad x\in \Sigma_{\lam_0}.
\end{equation}

The proof of this fact runs similarly as the second step of the proof of Theorem \ref{symmball}: suppose that \eqref{aprobar} does not hold, then by Theorem \ref{maxpplehypbdd}
$$
w_{\lam_0}(x)>0,\quad \forall x\in \Sigma_{\lam_0}.
$$

Observe that, by definition of $\lam_0$, there is a sequence $\lam_j\searrow \lam_0$ and $x_j \in \Sigma_{\lam_j}$ such that
\begin{equation} \label{eq-min}
w_{\lam_j}(x_j)=\min_{\Sigma_{\lam_j}}w_{\lam_j}<0, \quad \text{and}\quad \nabla w_{\lam_j}(x_j)=0.
\end{equation}

From condition \eqref{cond-f} we can guarantee that, up to a subsequence, $\{x_j\}_{k\in\N}$ converges to some $\bar x$. In fact, using \eqref{cond-decay}, if $|x_j|$ is sufficiently large we have that $u(x_j)$ is small and hence $\xi_{\lam_j}(x_j)$ is small, which implies that $f'(\xi_{\lam_j}(x_j))\leq 0$ due to \eqref{cond-f}. This analysis yields
\begin{align*}
\gfls u(x_j^{\lam_j})-\gfls u(x_j)&= f(u( x_j^{\lam_j}))-f(u(x_j))\\
&=f'(\xi_{\lam_j}(x_j)) w_{\lam_j}(x_j)\geq 0,
\end{align*}
which contradicts that fact that $x_j$ is a negative minimum of $w_{\lam_j}$ since  by the analysis derived in \eqref{es.negativo} we should have $\gfls u(x_j^{\lam_j})-\gfls u(x_j)<0$. Therefore, $\{x_j\}_{j\in\N}$ must be bounded.

Finally, from \eqref{eq-min}, $w_{\lam_0}(\bar x)\leq 0$ and then $\bar x\in T_{\lam_0}$, and $\nabla w_{\lam_0}(\bar x)=0$. Then, as in Theorem \ref{symmball} we deduce that
$$
\lim_{j\to\infty} \frac{w_{\lam_j} (x_j)}{\delta_j}=0, \qquad \delta_j:=\textrm{dist}(x_j,\Sigma_{\lam_j})
$$
which, in light of \eqref{equacion}, contradicts Proposition \ref{bdryest} and gives the result.
\end{proof}

\section{Extensions and applications}\label{sec.ext}

In this section we present some extensions, applications and further discussions of our results that we consider to be of interest. We start by pointing out that when the Young function $G$ is given by a power, i.e. $G(t)=\frac{t^p}{p},\:p\geq2$ we recover the known results for the fractional $p-$Laplacian, thus our results can be consider a generalization or extension of these to the nonhomogeneous scenario. Furthermore, as mentioned in the Introduction our setting allows for more general growth conditions such as $G(t)=t^p\log(1+t)$. 

Another special type of Young function that falls into the category studied here is 
\[
G(t)=t^p+t^q
\]
for $q>p\geq 2$. This structure is closely related to a special type of problems referred to as \emph{double phase variational problems} where the aim is to study minimizers of the functional
\begin{equation}\label{eq.double}
I(u)=\int_{\Omega}\left(|\nabla u|^p+a(x)|\nabla u|^q\right)\:dx;
\end{equation}
these have attracted much interest in the PDE community since the seminal work of Colombo and Mingione \cite{colombo2015regularity}. Our work is a step in the direction of obtaining qualitative properties of solution of the fractional analog of \eqref{eq.double}.

In a different direction, it is worth to mention that all the results stated in this manuscript hold true for a more general operators of the form
\[
\mathcal{L}_{g,s} (u):=\textrm{p.v.}\int_{\R^n}g\left(\frac{u(x)-u(y)}{k(|x-y|)}\right)\frac{dy}{k(|x-y|)|x-y|^n}
\]
where, for fixed constants $c_1, c_2>0$, $k\colon \R^n\times\R^n\times\R \to \R$ is such that
$$
c_1t^s \leq k(t)\leq c_2 t^s
$$
for any $t\geq 0$.

Finally, we would like to point out that the study of the qualitative properties discussed in this manuscript appears to be lacking in the literature for the \emph{local} case, that is operators of the form
\[
\text{div}\left(g(|\nabla u|)\frac{\nabla u}{|\nabla u|}\right).
\]
In that regard, a rather intriguing question is whether such local results could be recovered as a limit as $s\rightarrow 1^+$.

\subsection*{Acknowledgements.} This work was partially supported by Consejo Nacional de Investigaciones Cient\'ificas y T\'ecnicas. (CONICET).

\end{document}